\documentclass[12pt]{article}

\usepackage{amsmath}
\usepackage{amssymb}
\usepackage[mathscr]{eucal}
\usepackage{graphicx}
\usepackage{color}

\newtheorem{Theorem}{\sc Theorem}
\newtheorem{Definition}[Theorem]{\sc Definition}

\newtheorem{Lemma}[Theorem]{\sc Lemma}

\newtheorem{Remark}[Theorem]{\sc Remark}
\newtheorem{Example}[Theorem]{\sc Example}

\newlength{\sperr}
\def\BBox{\hbox{\vrule height 6pt depth 0pt width 6pt}}
\newenvironment{proof}{{\settowidth{\sperr}{\rm Proof}
		\hspace{-0.6cm}\parbox[t]{1.3\sperr}{\bf Proof.}
	}~}{\nopagebreak\mbox{}\hfill$\BBox$\par\addvspace{4mm}}
\newcommand{\bx}{\mbox{\boldmath{$x$}}}
\newcommand{\bzero}{\mbox{\boldmath{$0$}}}

\newcommand{\fb}{\mbox{\boldmath{$f$}}}
\newcommand{\bvarepsilon}{\mbox{\boldmath{$\varepsilon$}}}

\newcommand{\bu}{\mbox{\boldmath{$u$}}}

\newcommand{\bh}{\mbox{\boldmath{$h$}}}
\newcommand{\bs}{\mbox{\boldmath{$s$}}}

\newcommand{\bmu}{\mbox{\boldmath{$\mu$}}}
\newcommand{\bg}{\mbox{\boldmath{$g$}}}
\newcommand{\bp}{\mbox{\boldmath{$p$}}}
\newcommand{\bv}{\mbox{\boldmath{$v$}}}
\newcommand{\bphi}{\mbox{\boldmath{$\phi$}}}
\newcommand{\bw}{\mbox{\boldmath{$w$}}}
\newcommand{\bz}{\mbox{\boldmath{$z$}}}
\newcommand{\beeta}{\mbox{\boldmath{$\eta$}}}
\newcommand{\bsigma}{\mbox{\boldmath{$\sigma$}}}
\newcommand{\bomega}{\mbox{\boldmath{$\omega$}}}

\newcommand{\bzeta}{\mbox{\boldmath{$\zeta$}}}
\newcommand{\btau}{\mbox{\boldmath{$\tau$}}}

\newcommand{\bnu}{\mbox{\boldmath{$\nu$}}}

\newcommand{\dbu}{\mbox{\boldmath{$\delta u$}}}
\newcommand{\dom}{\mbox{\boldmath{$\delta \omega$}}}
\newcommand{\dbh}{\mbox{\boldmath{$\delta h$}}}
\newcommand{\dz}{\mbox{\boldmath{$\delta z$}}}
\newcommand{\eps}{\bvarepsilon}
\newcommand{\df}{\mbox{\boldmath{$\delta f$}}}

\newcommand{\RR}{\mathcal{R}}
\newcommand{\KK}{\mathcal{K}}
\newcommand{\GG}{\mathcal{G}}

\newcommand{\CC}{\mathcal{C}}
\newcommand{\OO}{\mathcal{O}}
\newcommand{\WW}{\mathcal{W}}
\newcommand{\BB}{\mathcal{B}}
\newcommand{\FF}{\mathcal{F}}

\newcommand{\R}{\mathbb{R}}

\newcommand{\MM}{\mathcal{M}}
\newcommand{\dual}[2]{\langle #1 , #2 \rangle}

\newcommand{\ve}{\mbox{{$\varepsilon$}}}

\newcommand{\cS}{\mbox{{${\cal S}$}}}
\newcommand{\cO}{\mbox{{${\cal O}$}}}
\newcommand{\cR}{\mbox{{${\cal R}$}}}
\newcommand{\cQ}{\mbox{{${\cal Q}$}}}
\newcommand{\cP}{\mbox{{${\cal P}$}}}

\newcommand{\cL}{\mbox{{${\cal L}$}}}

\def\sqr#1#2{{
		\vcenter{
			\vbox{\hrule height.#2pt
				\hbox{\vrule width.#2pt height#1pt \kern#1pt
					\vrule width.#2pt
				}
				\hrule height.#2pt
			}
		}
}}

\def\bar{\overline}
\def\real{\mathbb{R}}

\def\lista#1
{{ \itemindent 0.0cm \labelsep .2cm \leftmargin 0.8cm \rightmargin
		0.0cm \labelwidth 0.6cm \topsep 0.0mm
		\parsep 0.0mm
		\itemsep 0.0mm
		\begin{list}{}
			{ \setlength{\leftmargin}{0.8cm} \setlength{\rightmargin}{0.8cm}
				\setlength{\parsep}{0.0mm} \setlength{\topsep}{.0mm}
				\setlength{\parskip}{.0cm} \setlength{\itemsep}{.0cm} }
			{#1}\end{list}} }

\newcounter{theorem}

\title{\bf  Sensitivity analysis of a Signorini-type history-dependent  variational inequality}

\vspace{32mm}
{\author{Livia Betz$^1$\footnote{Corresponding author}\ , Andaluzia Matei$^2$\ and\ Mircea Sofonea$^3$\\[7mm]
{\it \small  $1$ Institute of Mathematics}\\
{\it \small  University of W\"urzburg }\\
{\it \small Emil-Fischer Str.\,30, 97074 W\"urzburg, Germany} \\ [7mm]		
{\it \small  $2$ Department of Mathematics}\\
{\it \small University of Craiova}\\
{\it \small Str. A.I. Cuza 13, Craiova, Romania}
\\	[5mm]
{\it \small  $3$ Laboratoire de Math\'ematiques et Physique}\\
{\it \small
	University of Perpignan Via Domitia}
\\{\it\small 52 Avenue Paul Alduy, 66860 Perpignan, France}		}

\date{}

\begin{document}
\maketitle

\vskip 4mm

\noindent {\small{\bf Abstract.}
We consider a history-dependent variational inequality $\cP$  which models the frictionless contact between a viscoelastic body and a rigid obstacle covered by a layer of soft material. The inequality is expressed in terms of the displacement field, is governed by the data $\fb$ (related to the applied body forces and surface tractions) and, under appropriate assumptions, it has a unique solution, denoted by $\bu$.   Our aim in this paper is to perform a sensitivity analysis of the inequality $\cP$, including the study of the regularity of the
solution operator $\fb\mapsto\bu$.
To this end, we start by proving the equivalence of $\cP$ with a fixed point problem, denoted by $\cQ$ (Theorem \ref{t1n}). We then consider an associated optimal control problem, for which we present an existence result (Theorem \ref{t2}). Then, we prove the directional differentiability of the solution operator and show that the directional derivative at $\fb$ in direction $\df$ is characterized by a history-dependent variational inequality with time-dependent constraints (Theorem \ref{t3}).  Finally, we prove two well-posedness results in the study of  Problems $\cP$ and $\cQ$, respectively  (Theorem \ref{t4}), and compare the two well-posedness concepts employed.

}


\vskip 4mm


\vspace{8mm}

\section{Introduction}\label{s1}
\setcounter{equation}0

Contact  between deformable bodies abounds in industry and everyday life. A few simple examples are brake pads in contact
with wheels, tires on roads, and pistons with skirts. Because
of the importance of contact processes in structural and mechanical
systems, considerable effort has been put into their modelling, analysis
and numerical simulations, and the literature in the field is extensive. It includes the
books~\cite{C,EJK, P, SM2, SofMig}, for instance.

In a weak formulation, most of the mathematical models of contact lead to variational inequalities. Their analysis is carried out by using arguments of monotonicity, compactness, fixed point, including the properties of the subdifferential of a convex function. It concerns results on existence, uniqueness and numerical approximation, including error estimates and convergence of discrete and semidiscrete schemes. Comprehensive references
in the field are the books \cite{G, HHNL, KO, P}, for instance.  In the last two decades, a growing attention was given to a special class of variational inequalities, namely the history-dependent variational inequalities. These are inequalities governed by a so-called history-dependent operator. An example is given by the Volterra integral operator and the memory operator which arises in the constitutive law of viscoelastic  materials.
Even if existence, uniqueness and convergence  results have been obtained in the literature (see \cite{SM2,SofMig}, for instance), there are quite few results on the optimal control, regularity of  solutions and well-posedness of  history-dependent variational inequalities. Concerning optimality conditions for the control of history-dependent variational inequalities, the literature is rather scarce and the only reference known are \cite{jnsao,aos,bp}.
The lack of contributions on the topic arises from the strong nonlinearity and nonsmoothness
of these problems, which leads to various mathematical difficulties.  However, the need to perform an in-depth sensitivity analysis of such inequalities and, in particular, of those inequalities which model contact phenomena, is widely recognized in the literature.

In the current paper we intend to fill this gap.
Thus, we consider a history-dependent variational inequality $\cP$ which describes the contact of a viscoelastic body with a foundation. The inequality  has as unknown the displacement field, is governed by a set of unilateral constraints, a history-dependent operator and a function $\fb$, which describes the applied body forces and surface tractions.  It has a special structure, since we show that it is equivalent with a fixed point problem, denoted in what follows by $\cQ$.
Under appropriate conditions on the data, the inequality has a unique solution $\bu$.  Our aim in this paper is three fold, as described  below.

The first one is to state an optimal control problem associated to the variational inequality $\cP$, denoted by $\cO$,  and to show the existence of optimal controls.
To this end we use arguments of compactness and lower semicontinuity.  Even if the  arguments are standard, the  novelty of the existence result we  present here arises from the special structure of the inequality $\cP$, whose controllability seams to be investigated here for the first time. The optimal control of a Signorini-type variational inequality with normal compliance was addressed in \cite{bst}. Nevertheless, the problem considered there was elliptic and the physical setting was one-dimensional. By contrast, in the current paper we consider a $d$-dimensional problem $(d\in\{1,2,3\}$) which, in addition, is governed by a history-dependent operator.
Further results on optimal control for variational inequalities  in Contact Mechanics can be found in   \cite{ACF02,CMMN,HZ19,MM11,MM17,MMN,ZMK21} and the references therein.

Our second aim is to study the regularity of the solution operator $\fb\mapsto \bu$ in terms of  directional differentiability. The fact that variational inequalities have limited differentiability properties is  a well-known fact. Starting with the pioneering contribution \cite{mp76}, there have been many advances on the topic in the last decades. From the numerous works that address the sensitivity analysis of variational inequalities, we only refer to\,\cite{by18, cc,Do,haraux,conical} and the references therein. However,  much less is known about the differentiability properties of variational inequalities  in the presence of a  history-dependent term. To the best  of our knowledge, the only papers dealing with this issue are \cite{jnsao} and \cite{aos}, where the sensitivity analysis of  history-dependent viscous damage evolutions is discussed. We point out that the variational inequality in the present paper  brings out some additional challenges. Besides having a long-term memory term, it displays a doubly non-smooth character: one is due to the inequality, while the second one is owed to the presence of a  so-called normal compliance term, see \eqref{pp} below.
To deal with these difficulties we rewrite our variational inequality appearing in problem $\cP$ as an identity, that is, as the fixed point problem $\cQ.$ The idea that variational inequalities can be equivalently written in this manner has already been successfully employed in various contributions dealing with optimal control of viscous damage models \cite{st_coup,jnsao, aos,  susu_dam}. In the present work, the arising fixed point mapping involves the solution operator of an elliptic variational inequality governed by a strongly monotone Lipschitz continuous operator. For this type of problems, directional differentiability has been investigated in \cite{ar, w0}
and we shall make use of these findings to  prove that our solution operator $\fb\mapsto \bu$ is Hadamard directionally differentiable. Its directional derivative at any point $\fb$ in direction $\df$ satisfies a history-dependent variational inequality with time-dependent constraints. This result  represents a second trait of novelty of our current work.

Finally, our third aim in this paper is  to study the well-posedness  of the Problems $\cP$ and $\cQ$.
 Here, the novelty arises from the fact that  we introduce two different well-posedness concepts, compare them, and prove the corresponding well-posedness results.
Well-posedness concepts in the study of nonlinear problems, including variational and hemivariational  inequalities, optimization  and fixed point problems have been intensively studied in the last decades. Basic references in the field are \cite{DZ,L} and, more recently, \cite{S}.

The rest of the manuscript is structured as follows. In Section \ref{s2} we introduce some notation, state the inequality problem $\cP$, list the assumptions on the data and recall an existence and uniqueness result in the study of Problem $\cP$ (Theorem \ref{t1}). Moreover, we provide some mechanical interpretations. In Section \ref{s2n} we introduce the fixed point problem $\cQ$ and prove that it is equivalent with the inequality problem $\cP$. Then, in Section \ref{s3} we focus on the associated optimal control problem  $\cO$ and prove an existence result (Theorem \ref{t2}). The sensitivity analysis of the solution operator for Problems $\cP$ and $\cQ$ is provided in Section \ref{s4}. There, we state and prove our main result in this paper (Theorem \ref{t3}). The proof of this theorem is carried out in several steps, based on some intermediate results.  We end our paper with Section \ref{s5}, dedicated to the well-posedness of the Problems $\cP$ and $\cQ$. There, we prove two well-posedness result (Theorem \ref{t4}), and discuss about the well-posedness concepts we consider.

\section{Problem statement}\label{s2}
\setcounter{equation}0

This preliminary section is structured into three parts, as follows.

\medskip\noindent
{\bf Notation.}  We use the symbols ``$\to$" and ``$\rightharpoonup$" to indicate the strong and weak convergence in various normed spaces that will be specified, except in the case when these convergences take place in $\R$. All the limits and lower limits are considered as $n\to\infty$, even if we do not mention it explicitly. For a sequence $\{\ve_n\}\subset\R_+$ which converges to zero we use the short hand notation $0\le \ve_n\to 0$.

Let $\mathbb{S}^d$  be  the space of second order symmetric
tensors on $\mathbb{R}^d$ with $d\in\{1,2,3\}$.
We  denote by $``\cdot"$ and ``$\| \cdot\|$" the inner product and the Euclidean norm on the spaces  $\mathbb{R}^d$  and  $\mathbb{S}^d$, respectively.
We also consider a bounded domain $\Omega\subset \real^d$ with smooth boundary, denoted by  $\Gamma$.  Moreover, we assume that $\Gamma_1$ is a measurable part of $\Gamma$
whose  $d-1$ measure, denoted by ${meas}\,(\Gamma_1)$, is positive. The outward unit normal at $\Gamma$ will be denoted by $\bnu$ and $[0,T]$ will represent the time interval of interest, with $T>0$ fixed.   In addition, $\bx$ stands for  a typical point in $\Omega\cup\Gamma$  and, for simplicity, we sometimes skip the dependence of various functions on the spatial variable $\bx$, that is, we write $\bu(t)$ instead of $\bu(x,t)$.

We use the standard notation for the Lebesgue and Sobolev spaces associated to $\Omega$, $\Gamma$ and $T$. Typical examples are the spaces
$L^2(\Omega)^d$, $L^2(\Gamma)^d$, $H^1(\Omega)^d$ and $H^1(0,T;L^2(\Omega)^d)$, equipped with their canonical Hilbertian structure.   For an element $\bv\in H^1(\Omega)^d$ we still write $\bv$ for the trace $\gamma\bv\in L^2(\Gamma)^d$ and $v_\nu\in L^2(\Gamma)$ for the normal trace on the boundary, i.e.,
$v_\nu=\bv\cdot\bnu$. Moreover,
$\bvarepsilon(\bv)$ will denote the symmetric
part of the gradient of $\bv$, i.e.,
\begin{equation}\label{de}
\bvarepsilon(\bv)=\frac{1}{2}\big(\nabla \bv+\nabla\bv^T\big).
\end{equation}

Next, we introduce the spaces $V$ and  $Q$, defined as follows:
\begin{eqnarray}
	&&\label{spV}V=\{\,\bv=(v_i)\in H^1(\Omega)\ :\  v_i =0\ \ {\rm on\ \ }\Gamma_1 \ \ \ \forall\, i=\overline{1,d}\,\},\\
	&&\label{spQ}Q=\{\,\bsigma=(\sigma_{ij}):\ \sigma_{ij}=\sigma_{ji} \in L^{2}(\Omega)\ \ \ \forall\, i,\,j=\overline{1,d}\,\}.
\end{eqnarray}
The spaces $V$ and $Q$  are real Hilbert spaces endowed with the inner products
\begin{equation}\label{eeq}
	(\bu,\bv)_V= \int_{\Omega}
	\bvarepsilon(\bu)\cdot\bvarepsilon(\bv)\,dx,\qquad ( \bsigma,\btau )_Q =
	\int_{\Omega}{\bsigma\cdot\btau\,dx}.
\end{equation}
The associated norms on these spaces will be denoted by $\|\cdot\|_{V}$ and $\|\cdot\|_{Q}$, respectively.  Recall that the completeness of the space $(V,\|\cdot\|_{V})$ follows from the
assumption ${meas}\,(\Gamma_1)>0$, which allows the use of Korn's
inequality.
Note also that, by the definition of the inner product in the spaces $V$ and $Q$, we have
\begin{equation}
	\|\bv\|_V=\|\bvarepsilon(\bv)\|_Q\qquad\forall\,\bv\in V
	\label{684}
\end{equation}
and, using the Sobolev
theorem, we deduce that
\begin{equation}\label{trace}
	\|\bv\|_{L^2(\Gamma)^d}\le c_{0}\,\|\bv\|_{V}\qquad
	\forall\,\bv \in V.
\end{equation}
Here $c_{0}$ is a positive constant which depends on $\Omega$ and $\Gamma_1$. We also denote by $V^*$ the dual of the space $V$ and
$\dual{ \cdot}{\cdot}$ will represent the duality pairing between $V$ and $V^*$.

In addition,
we need the space of fourth order tensors
${\bf Q_\infty}$ given by
\begin{equation}\label{Qi} {\bf Q_\infty}=\{\, {\cal D}=(D_{ijkl})\ : \
	{D}_{ijkl}={D}_{jikl}={D}_{klij} \in L^\infty(\Omega) \ \ \forall\, i,\,j,\,k,\,l=\overline{1,d}\,\}\,.
\end{equation} It is easy to see that ${\bf Q_{\infty}}$ is
a real Banach space with the norm
\begin{equation*}\label{**}
	\displaystyle \|{\cal{D}}\|_{\bf Q_{\infty}}=\max_{0\le i,j,k,l\le	d}\|{D}_ {ijkl}\|_{L^{\infty}(\Omega)}
\end{equation*} \noindent
and, moreover,
\begin{equation}\label{pmp}
	\|{\cal{D}}\btau\|_{Q}\le d\,\|{\cal{D}}\|_{\bf Q_{\infty}}
	\|\btau\|_Q\quad \ \forall\,
	{\cal{D}}\in{\bf Q_{\infty}},\ \btau\in Q.
\end{equation}

For any normed space
$(X,\|\cdot\|_X)$ we denote by $C([0,T];X)$  the space of continuous functions defined on $[0,T]$ with values in $X$. It is well known that if $X$ is a Banach space, then
$C([0,T];X)$ is a Banach space,
equipped with the following norm
\begin{equation*}
	\|\bv\|_{C([0,T];X)} =
	\max_{t\in [0,T]}\,\|\bv(t)\|_X\qquad\forall\, \bv\in C([0,T];X).
\end{equation*}
Moreover, if $K\subset X$, we denote by
  $C([0,T];K)$ the set of functions in
 $C([0,T];X)$ with values in $K$.

 Let now  $(X,\|\cdot\|_X)$  and $(Y,\|\cdot\|_Y)$ be two normed spaces. We recall that  an operator $\Lambda:C([0,T],X)\to C([0,T],Y)$ is said to be a history-dependent operator if there exists $L>0$ such that
\begin{equation*}
	\|\Lambda \bv(t)-\Lambda\bw(t)\|_Y\le L\int_0^t\
	\|\bv(s)-\bw(s)\|_X\,ds\qquad\forall\,\bv,\, \bw\in C([0,T],X),\ t\in[0,T].
\end{equation*}
Introduced in \cite{SM-H}, history-dependent operators have
a number of useful properties, that can be found in the books \cite{SM2,SofMig}. Part of these properties will play a key role in the next sections of this paper.

\medskip\noindent
{\bf The inequality problem $\cP$.} The variational inequality we consider in this paper is  stated as follows.

\medskip\noindent
{\bf Problem $\cP$.} {\it
Find a function $\bu:
[0,T]\to V$ such that, for all $t\in
[0,T]$, the inequality below holds:}
\begin{eqnarray}
	&&\label{si} \bu(t)\in
	U,\quad(\BB\bvarepsilon({\bu}(t)),\bvarepsilon(\bv)-\bvarepsilon(\bu(t)))_Q\\[2mm]
	&&\qquad+
	\Big(\int_0^t\RR(t-s)\bvarepsilon({\bu}(s))\,ds,\bvarepsilon(\bv)-\bvarepsilon(\bu(t))\Big)_Q\nonumber
	\\[2mm]
	&&\qquad\qquad+\dual{P\bu(t)}{\bv-\bu(t)}\geq
	\dual{ \fb(t)}{\bv-\bu(t)}\quad \forall\,\bv\in U.\nonumber
\end{eqnarray}

\medskip
This problem is governed by the data $U$, $\BB$, $\RR$, $P$, $\fb$, assumed to satisfy the following conditions.
\begin{eqnarray}
&&\label{U} \mbox{ $U$ is a nonempty closed convex subset of $V$}. \\ [3mm]
&&\label{B}\left\{\begin{array}{ll}\BB :Q\to Q\ \mbox{is a linear continuous operator and}
	\\[1mm]
	\mbox{there exists $m_{\mathcal B}>0$ such that}\
	(\BB\btau,\btau)_Q\geq m_{\mathcal B}\,\|\btau\|^2_Q\ \ \forall\,\btau\in Q.	
\end{array}\right.
\end{eqnarray}
\begin{eqnarray}
&&\label{R}\RR \in C([0,T];{\bf Q_\infty}).\\ [3mm]
&&\label{P}\mbox{ $P: V\to V^*$ is a Lipschitz continuous monotone operator}.\\ [3mm]
&&\label{f} \fb\in C([0,T];V^*).
\end{eqnarray}

\medskip
Note that the second term in the variational inequality \eqref{si} includes a history-dependent operator. Moreover, with a particular choice of the data, the inequality models a Signorini viscoelastic contact problem, as we explain below in this section. For this reason, we refer to Problem $\cP$ as a Signorini-type history-dependent variational inequality.
Its unique solvability is provided by the following existence and uniqueness result.
\begin{Theorem}\label{t1}
	Assume \eqref{U}--\eqref{f}. Then, Problem $\cP$ has a unique solution  $\bu\in C([0,T];U)$.
\end{Theorem}

A proof of Theorem \ref{t1} can be found in
\cite[p.180]{SM2}. This is based on standard arguments on elliptic variational inequalities and a fixed point property of history-dependent operators.

\medskip\noindent
{\bf Mechanical interpretation.} Problem $\cP$ is inspired by a mathematical model of contact that we describe in what follows. Consider a viscoelastic body which occupies, in its reference configuration, the domain $\Omega$. Besides  $\Gamma_1$, we now assume that $\Gamma_2$ and $\Gamma_3$ are given measurable parts of $\Gamma$ such that $\Gamma=\bar{\Gamma}_1\cup\bar{\Gamma}_2\cup\bar{\Gamma}_3$ and $\Gamma_i\cap\Gamma_j
=\emptyset$, for $i,\ j=1,\, 2,\ 3$, $i\ne j$.
The body is fixed on the  part $\Gamma_1$ of its boundary, is acted by a time-dependent surface traction on $\Gamma_2$ and is in frictionless contact with an obstacle on $\Gamma_3$, the so-called foundation. The foundation is made of a rigid body covered by a layer of deformable material. Moreover, a time-dependent body force is acting on the viscoelastic body. We model the material's behaviour with a linear constitutive law with long term memory and  assume that the process is quasistatic, that is, we neglect the acceleration term in the equation of motion.

\smallskip
 The classical formulation of the mechanical problem described above  is given by a system of partial differential equations, obtained by gathering the constitutive law, the equilibrium equation, the displacement and traction boundary condition, as well as the frictionless contact condition.  In a weak  formulation, the model leads to a variational inequality of the form \eqref{si}
in which the unknown $\bu$  represents  the displacement field
and, hence, $\bvarepsilon(\bu)$ is the linearized strain tensor. The details can be found in \cite{SM2,SofMig} and, therefore, we skip them. We just restrict ourselves to mention
the following.

\smallskip
Firstly,  the operator $\BB$ is given by the equality
\[(\BB\bsigma,\btau)_Q=\int_\Omega\BB_{ijkl}\sigma_{ij}\tau_{kl}\,dx\qquad \forall\, \bsigma=(\sigma_{ij}),\, \btau=(\tau_{ij})\in Q,\]
where the  components $\BB_{ijkl}\in L^\infty(\Omega)$ satisfy the usual symmetry and coercivity condition. Moreover,  $\RR$ is the relaxation tensor which describes the memory behavior of the material.

\smallskip
The set $U$ represents the set of admissible displacements fields  given by
\begin{equation}\label{eq:convexset}
	U=\{\,\bv\in V:\, \bv_\nu\le g\ \ {\rm a.e.\ on\ }\Gamma_3\,\},
\end{equation}
where $g\in L^2(\Gamma_3)$ is a positive function which measures the thickness of the deformable material.
Condition $\bu(t)\in  U$ shows that the normal displacement is restricted on the potential contact surface. It arises as a consequence of the Signorini nonpenetration condition  we use in order to model the reaction of the rigid part of the foundation.

\smallskip
By contrast, the deformable layer  allows penetration and, therefore, we model its behavior with the so-called normal compliance condition. This assumption {gives} rise to the operator $P: V\to V^*$ in $\eqref{si}$, defined by the equality
\begin{equation}\label{pp}
	\dual{P\bu}{\bv}=\int_{\Gamma_3} p(\bu_\nu)\bv_\nu\,da\quad\
	\forall\,\bu,\,\bv\in V,
\end{equation} where $p$ is a given compliance function assumed to satisfy the following conditions:
\begin{equation}
		\left\{\begin{array}{ll} {\rm (a)\ }
			p:\Gamma_3 \times\mathbb{R}\to\mathbb{R}_+.\\ [1mm]
			{\rm (b)\ There\ exists}\, L_p>0 {\rm\ such\ that\ }\\
			{} \qquad  |p(\bx,r_1)-p(\bx,r_2)|
			\le L_p\,|r_1-r_2|\\
			{}\qquad\quad \forall\,
			r_1,\,r_2\in \mathbb{R},\ {\rm a.e.\ } \bx\in\Gamma_3.\\ [1mm]
			{\rm (c)\ }(p(\bx,r_1)-p(\bx,r_2))(r_1-r_2)\ge 0\\
			{}\qquad\quad \forall\,
			r_1,\,r_2\in \mathbb{R},\ {\rm a.e.\ } \bx\in\Gamma_3.\\ [1mm]
			{\rm (d)\ The\ mapping\ } \bx\mapsto
			p(\bx,r) {\rm\ is\ measurable\ on}\ \Gamma_3,
			\\
			{}\qquad{\rm for\ any\ }r\in \mathbb{R}.\\ [1mm]
			{\rm (e)\ }p(\bx,r)=0\ {\rm for\ all\ }r\le 0,\ {\rm a.e.}\
			\bx\in\Gamma_3.\\ [1mm]
			{\rm (f)\ }{\rm The\ mapping\ } r\mapsto
			p(\bx,r) {\rm\ is\ semidifferentiable\ on}\ \R,\  {\rm a.e.}\
			\bx\in\Gamma_3.
		\end{array}\right.
		\label{G7n}
	\end{equation}
We recall that a function is semidifferentiable at a point  if it has a left- or right derivative there. That is, Assumption \eqref{G7n}(f) accounts to
the existence of at least one of the following limits
\[\lim_{\tau \searrow 0}\frac{p(\bx,r+\tau)-p(\bx,r)}{\tau} \]
\[\lim_{\tau \searrow 0}\frac{p(\bx,r)-p(\bx,r-\tau)}{\tau} \]for each $r\in \R,{\rm a.e.}\ \bx \in \Gamma_3.$ This is equivalent to the directional differentiability of $p(\bx,\cdot):\R \to \R \ {\rm a.e.}\ \bx \in \Gamma_3,$ see Definition \ref{def:dd}.a) below.

Assumption \eqref{G7n} combined with the definition \eqref{pp} and the trace inequality \eqref{trace} guarantees that the operator $P$ satisfies condition \eqref{P}. Moreover, a classical example of function $p$ which satisfies \eqref{G7n} is given by $p(\bx,r)= cr_+$ for all $\bx\in \Gamma_3$ and $r\in\mathbb{R}$, where $c>0$ denotes a stiffness coefficient and $r_+$ is the positive part of $r$.

Finally, the function $\fb$ in \eqref{si} is given by
\begin{equation}\label{cop}\dual{ \fb(t)}{\bv}=\int_{\Omega} \fb_0(t)\cdot\bv\,dx+\int_{\Gamma_2}\fb_2(t)\cdot \bv\,da\quad\ \forall\, t\in[0,T], \quad \forall\,\bv \in V,
\end{equation}
where ${\fb}_0$ and ${\fb}_2$  denote the densities of the body forces and surface tractions, respectively.
 We easily observe that if
\begin{equation}\label{ipof0f2}
\fb_0\in H^1(0,T;L^2(\Omega)^d),\quad \fb_2\in  H^1(0,T; L^2(\Gamma_2)^d),
\end{equation}
then the condition \eqref{f} is satisfied.

\section{Equivalent formulation}\label{s2n}
\setcounter{equation}0

In this section we keep the assumption in Theorem \ref{t1} and provide an equivalent formulation of Problem $\cP$, which is  in a form of a fixed point problem. To this end we introduce in what follows  some  useful operators.

First, we consider the operator $\bvarepsilon:V \to Q$
defined by \eqref{de} and denote by  $\bvarepsilon^*: Q \to V^*$ its adjoint. Then,
\begin{equation}\label{es}
	\dual{\bvarepsilon^*(\beeta)}{\bv} = (\beeta,\bvarepsilon(\bv))_Q \quad\ \forall\,\beeta\in Q,
	\,\bv \in V
\end{equation}
and, therefore,
\begin{equation}\label{esn}
\|\bvarepsilon^*(\beeta)\|_{V^*}\le
\|\beeta\|_{Q} \quad\ \forall\,\beeta\in Q.
\end{equation}
We shall use inequality \eqref{esn} in various places in the rest of the manuscript.

Next, we define the operator $\WW:V\to V^*$ by the equality
\begin{equation}\label{ww}
\WW=\bvarepsilon^* \BB \bvarepsilon.
\end{equation}
Using \eqref{es}, \eqref{ww}, assumption \eqref{B} and the equality \eqref{684}, it follows that
\begin{eqnarray}
&&\label{wa}\dual{\WW \bu}{\bv}=(\BB\bvarepsilon(\bu),\bvarepsilon(\bv))_Q
\quad\ \ \forall\, \bu,\, \bv \in V,\\ [2mm]	
&&\label{coerc}
\dual{\WW \bv}{\bv}=(\BB  \bvarepsilon(\bv), \bvarepsilon(\bv))_Q \geq m_{\BB} \| \bvarepsilon(\bv)\|_Q^2 =m_{\BB}\|\bv\|^2_V \quad \forall\,\bv \in V,
\end{eqnarray}
which shows that the operator $\WW$ is linear, continuous 	and coercive.

We now use the above properties of $\WW$, assumptions \eqref{U} and \eqref{P} on the set $U$ and the operator $P$, respectively, together  with standard arguments on elliptic variational inequalities. In this way we deduce that for any $\bomega\in V^*$ there exists a unique element $\bz\in V$ such that
\begin{equation*}
	\bz\in U,\quad  \dual{(\WW+P)(\bz)}{\bv-\bz}\geq \dual{ \bomega}{\bv-\bz}\quad \forall\,\bv\in U.\nonumber
\end{equation*}
We denote by $\FF:V^* \ni \bomega \mapsto \bz \in U$ the solution map of this inequality, that is,
\begin{equation}\label{eq}
\bz=\FF\bomega\ \Longleftrightarrow\ \bz\in U,\quad  \dual{(\WW+P)(\bz)}{\bv-\bz}\geq \dual{ \bomega}{\bv-\bz}\quad \forall\,\bv\in U.
\end{equation}
Moreover,  using standard arguments combined with  inequality \eqref{coerc}, it is easy to see that $\FF$ is a Lipschitz continuous operator, with constant $\frac{1}{m_{\mathcal B}}$, i.e.,
\begin{equation}\label{Fl}
\|\FF\bomega_1-\FF\bomega_2\|_V\le \frac{1}{m_{\mathcal B}}\,\|\bomega_1-\bomega_2\|_{V^*}\quad\ \forall\, \bomega_1,\, \bomega_2\in V^*.
\end{equation}	

With these preliminaries, we are in a position to consider the following problem.

\medskip\noindent
{\bf Problem $\cQ$.} {\it
	Find a function $\bu:
	[0,T]\to V$ such that, for all $t\in
	[0,T]$, the following equality holds:}
	\begin{equation}\label{fp}
		\bu(t)
			=   \FF\Big[ \fb(t)- \eps^*\Big(\int_0^t  \RR(t-s)\eps(\bu(s))\,ds\Big)\Big].
	\end{equation}

\medskip
The link between Problems $\cP$ and $\cQ$ is provided by the following equivalence result.

\begin{Theorem}\label{t1n} 	Assume \eqref{U}--\eqref{f} and let
$\bu\in C([0,T];U)$. Then, $\bu$ is a solution to problem $\cP$ if and only if $\bu$ is a solution of Problem $\cQ$.
\end{Theorem}

\begin{proof} Let $t\in[0,T]$. We use  \eqref{es}, \eqref{wa}, \eqref{eq} to see that the following equivalences hold:
\begin{eqnarray*}
&&\mbox{$\bu(t)$ satisfies inequality \eqref{si}} \Longleftrightarrow\ \\ [2mm]
&&\dual{(\WW+P)\bu(t)}{\bv-\bu(t)}\ge \dual{\fb(t)-\eps^*\Big(\int_0^t  \RR(t-s)\eps(\bu(s))\,ds\Big)}{\bv-\bu(t)} \Longleftrightarrow\ \\ [2mm]
&&\mbox{$\bu(t)$ satisfies equation \eqref{fp}},
\end{eqnarray*}
which concludes the proof.	
\end{proof}

Note that the importance of Problem $\cQ$ arises from the fact that it represents an alternative to Problem $\cP$. Indeed, in various estimates we present  below it will be more convenient to use the nonlinear nonsmooth equation \eqref{fp} instead of the variational inequality \eqref{si}. Moreover, note that
Theorems \ref{t1} and \ref{t1n} show that, under assumptions  \eqref{U}--\eqref{f}, Problem $\cQ$ has a unique solution with regularity $\bu\in C([0,T];U)$.

\begin{Remark}\label{r1}
Consider now the operator $\Lambda:  C([0,T];V)\to  C([0,T];U)$ defined by
\begin{equation}\label{la}
\Lambda\bu(t)= \FF\Big[ \fb(t)- \eps^*\Big(\int_0^t  \RR(t-s)\eps(\bu(s))\,ds\Big)\Big] \quad\ \forall\,\bu\in C([0,T];V),\, t\in[0,T].
\end{equation}
Then, it is easy to see that Problem $\cQ$ is equivalent with the fixed point problem of finding a function $\bu\in C([0,T];V)$ such that
\begin{equation}\label{ffp}
\bu=\Lambda\bu.
\end{equation}
Moreover, the equivalence result in Theorem $\ref{t1n}$ shows that, roughly speaking, Problem $\cP$ has a fixed point structure. We shall use this remark in Section $\ref{s5}$, in the study of well-posedness of Problems $\cP$ and $\cQ$.
\end{Remark}

Theorems $\ref{t1}$ and $\ref{t1n}$ allow us
to define  the control-to-state operator associated to Problems $\cP$ and $\cQ$. This will play an important role in the analysis we perform in Sections $\ref{s3}$ and $\ref{s4}$ below.

\begin{Definition}\label{r2}
The {\em  solution operator} of Problems $\cP$ and $\cQ$ is given by
$S: C([0,T];V^*)\to C([0,T];U) $  defined as
\begin{equation}\label{so}
	S(\fb)=\bu.
\end{equation}
This associates to any  function  $\fb\in C([0,T];V^*)$ the common solution
$\bu \in C([0,T];U)$ of Problems $\cP$ and $\cQ$. \end{Definition}

We now end this section with an elementary example.

\begin{Example}\label{ex0} Consider Problem $\cP$ in the particular case when $d=1$, $\Omega=(0,1)$, $\Gamma_1=\{0\}$,
$\Gamma_2=\emptyset$,	$\Gamma_3=\{1\}$,  $g=1$, $\BB\tau=\tau$ for all $\tau\in Q$, $\RR(t)=1$ for all $t\in[0,T]$, $Pv=0_{V^*}$   for all $v\in V$ and $f$ to be specified later. Note that in this particular case $\ve(u)=u'$,
where the prime represents the derivative with respect to the spatial variable $x$. Moreover,
\begin{eqnarray*}
	&&V=\{\, v\in H^1(0,1):\  v(0) =0\},\quad Q=L^2(0,1),\\ [2mm]
	&&U=\{\, v\in H^1(0,1):\  v(0) =0,\ \ v(1)\le 1\,\}.
\end{eqnarray*}
Then, it is easy to see that Problem $\cP$ consists of finding a function $u:[0,T]\to V$ such that, for all $t\in[0,T]$, the following holds:
\begin{equation}\label{41}
	u(t)\in U,\quad\Big(u'(t)+\int_0^t u'(s)\,ds,v'-u'(t)\Big)_Q\ge \dual{f(t)}{v-u(t)}\quad\forall\, v\in U.
\end{equation}
Take {$f\in C([0,T];V^*)$ as
$\dual{f(t)}{v}=(1,v')_Q$  for all $v\in V$, $t\in[0,T]$} and recall that
	\[(\sigma,\tau)_Q=\int_0^1 \sigma\tau\,dx\quad\ \forall\, \sigma,\ \tau\in Q.\]
Then,  \eqref{41} becomes
\begin{equation}\label{42}
	u(t)\in U,\quad \int_0^1\Big(u'(t)+\int_0^t u'(s)\,ds-1\Big)(v'-u'(t))\,dx\ge 0\quad\ \forall\, v\in U.
\end{equation}

Let $u:[0,1]\times[0,T]\to\R$ be the function given by
\begin{equation}\label{sol}
u(x,t)=xe^{-t}\quad\forall\,(x,t)\in [0,1]\times[0,T].
\end{equation}
It is easy to see that  $u(t)\in U$  and
\[u'(t)+\int_0^t u'(s)\,ds=1,\] for  any $t\in[0,T]$. This shows that $u$ satisfies inequality \eqref{42} and, therefore, it is the  solution to Problem $\cP$.
Moreover,  it is the solution to the corresponding  Problem $\cQ$, as it follows from the equivalence result in Theorem $\ref{t1n}$.

\end{Example}

\section{An optimal control problem}\label{s3}
\setcounter{equation}0
In this section we  study an optimal control problem associated to the variational inequality \eqref{si}.  To this end
we  assume in what follows that \eqref{U}--\eqref{P} hold. Moreover, we consider a function $j$, a real reflexive Banach space $Z$, an operator $\KK,$ a constant $\beta$ and a set $\MM$
which satisfy the following conditions:
\begin{eqnarray}
&&\label{z1}j:C([0,T];V) \to \R_+\ \mbox{is  lower semicontinuous,}\\ [2mm]
&&\label{z11a} \KK:H^1(0,T;Z) \to C([0,T];V^*) \mbox{ is completely continuous,}\\[2mm]
&&\label{z1n}\beta>0,\\ [2mm]
&&\label{z2} \mbox{$\MM$ is a weakly sequentially closed subset of the space  $H^1(0,T;Z)$.}
\end{eqnarray}

Note that the property \eqref{z11a} implies that
\begin{equation}\label{ce}\bphi_n \rightharpoonup \bphi \,\text{ in }\,H^1(0,T;Z)\ \Longrightarrow \ \KK(\bphi_n) \to \KK(\bphi)\, \text{ in }\,C([0,T];V^*).\end{equation}
This implication  will represent an important ingredient in the proof of Theorem \ref{t2} below. {Examples of a function $j$ and operators $\KK$ which satisfy conditions \eqref{z1} and \eqref{z11a}, respectively, will be presented at  the end of this section, together with the correponding mechanical interpretations.}

Let $\cL:\MM\to\R$ be the cost function defined by
\begin{equation}\label{z3}
\cL(\bg)=j(S(\KK(\bg)))+\beta\,\|\bg\|^2_{H^1(0,T;Z)}\quad\ 
{\forall\, \bg\in\MM},
\end{equation}
where we recall that $S$ represents the solution operator from \eqref{so}. Then, the optimal control problem we consider can be formulated as follows.

\medskip\noindent
{\bf Problem $\cO$.} {\it
Find $\bg^*\in\MM$ such that}
\begin{equation}\label{z4}
\cL(\bg^*)\le \cL(\bg)\qquad\forall\, \bg\in \MM.
\end{equation}

\medskip
Our main result in this section is the following existence result.

\begin{Theorem}\label{t2}Assume \eqref{U}--\eqref{P} and \eqref{z1}--\eqref{z2}. Then, Problem $\cO$ admits at least one solution.
\end{Theorem}

To provide the proof of this theorem we need the following
property of the operator $S$.

\begin{Lemma}\label{l2}
Assume \eqref{U}--\eqref{P}. Then, the solution operator defined in \eqref{so} is a Lipschitz continuous operator with Lipschitz constant $K=K(\mathcal B,d, \mathcal R,T)$, i.e.,
\begin{eqnarray*}
			\|S(\fb)-S(\widetilde \fb)\|_{C([0,T];V)}\leq K\|\fb-\widetilde \fb\|_{C([0,T];V^*)} \quad \forall\, \fb,\widetilde{\fb}\in C([0,T];V^*).
	\end{eqnarray*}
\end{Lemma}

	\begin{proof} Let $\fb,\, \widetilde \fb\in C([0,T];V^*)$ and  denote $S(\fb)=\bu$, $S(\widetilde \fb)=\widetilde \bu$.  Let $t\in[0,T]$.
	We use  equation \eqref{fp} and inequalities \eqref{Fl}, \eqref{esn}, \eqref{pmp}
	to see that
	\begin{eqnarray*}
			&& \|\bu(t)-\widetilde \bu(t)\|_V\leq \frac{1}{m_{\mathcal B}}\|\fb(t)-\widetilde \fb(t)\|_{V^*}\\
			&&\qquad+ \frac{d}{m_{\mathcal B}}\,\|\mathcal R\|_{C([0,T];{\bf Q_\infty)}}\int_0^t\|\bu(s)-\widetilde \bu(s)\|_V\,ds.
		\end{eqnarray*}
		Then, using  Gronwall's inequality we deduce that
		\begin{eqnarray}
		&&\label{z5}\|\bu(t)-\widetilde \bu(t)\|_V \leq\frac{1}{m_{\mathcal B}}\,\|\fb(t)-\widetilde \fb(t)\|_{V^*}\\[2mm]
		&&\qquad+ c \int_0^t {\frac{1}{m_{\mathcal B}}}\|\fb(s)-\widetilde \fb(s)\|_{V^*} e^{c(t-s)}\, ds\nonumber
		\end{eqnarray}
	where, here and below,  $c=\frac{d}{m_{\mathcal B}}\, \|\mathcal R\|_{C([0,T];{\bf Q_\infty})}$.
	Inequality \eqref{z5} implies that
	\begin{eqnarray*}
	&&	\|\bu(t)-\widetilde \bu(t)\|_V\leq\frac{1}{m_{\mathcal B}}\|\fb-\widetilde \fb\|_{C([0,T];V^*)}\\ [2mm]
	&&\quad+{\frac{c}{m_{\mathcal B}}}\,\|\fb-\widetilde \fb\|_{C([0,T];V^*)}\int_0^t e^{c(t-s)}\,ds\leq {\frac{1}{m_{\mathcal B}}\Big(1+e^{c t}-1\Big)\|\fb-\widetilde \fb\|_{C([0,T];V^*)}}
	\end{eqnarray*}
	and, therefore,
	\begin{eqnarray*}
			\|\bu-\widetilde \bu\|_{C([0,T];V)}\leq{ \frac{1}{m_{\mathcal B}}e^{c T}}\|\fb-\widetilde \fb\|_{C([0,T];V^*)}.
	\end{eqnarray*}
	We now take  $K(\mathcal B,d, \mathcal R,T)={\frac{1}{m_{\mathcal B}}e^{c T}}$ to conclude the proof.
	\end{proof}

We are now in the position to provide the proof of Theorem \ref{t2}.

\medskip

\begin{proof}
Denote by $L_0$ the lower bound of the cost function  $\cL$, that is
\begin{equation}\label{ii}
L_0=\inf_{\bg\in\MM}\cL({\bg})<+\infty.
\end{equation}
Let $\{{\bg}_n\}\subset\MM$ be a minimizing sequence, i.e.,
\begin{equation}\label{inf}
\cL(\bg_n)=	j(S(\KK(\bg_n)))+\beta\,\|\bg_n\|^2_{H^1(0,T;Z)}\to L_0.
\end{equation}
The convergence \eqref{inf} combined with assumption \eqref{z1} shows that the sequence $\{{\bg}_n\}$ is bounded in the reflexive Banach space $H^1(0,T;Z)$, whence we deduce that there exists an element $\bg^*\in H^1(0,T;Z)$ such that, passing to a subsequence again denoted by  $\{{\bg}_n\}$, we have
\begin{equation}\label{z6}
	\bg_n \rightharpoonup \bg{^*} \quad \text{ in }\ H^1(0,T;Z).
\end{equation}
On the other hand, inclusion $\{{\bg}_n\}\subset\MM$ combined with the weak convergence \eqref{z6} and assumption \eqref{z2}   imply that
\begin{equation}\label{z6n}
	\bg^*\in\MM.
\end{equation}
We now use \eqref{z6} and the property \eqref{ce} to deduce that
\[\KK(\bg_n) \to \KK(\bg^*) \quad \text{in }C([0,T];V^*).\]
Therefore, by the  continuity of the solution operator $S$, guaranteed by Lemma \ref{l2}, we find that
\begin{equation}\label{z11}S(\KK(\bg_n)) \to S(\KK(\bg^*)) \quad \text{in }\ C([0,T];V).
\end{equation}

We now use the strong convergence \eqref{z11}, assumption \eqref{z1}, and \eqref{z6} combined with the weak lower semicontinuity of the norm squared
to see that
\[j(S(\KK(\bg^*)))+\beta\|\bg^*\|^2_{H^1(0,T;Z)}\leq \liminf_{n \to \infty} \left(j(S(\KK(\bg_n)))+\beta\|\bg_n\|^2_{H^1(0,T;Z)}\right)\] and, therefore,  definition \eqref{z3} implies that

\begin{equation}\label{z12}
\cL(\bg^*)\le \liminf_{n \to \infty}\cL(\bg_n).
\end{equation}
Further, by \eqref{z12} and the convergence \eqref{inf} we see that
\begin{equation}\label{z12n}
	\cL(\bg^*)\le L_0.
\end{equation}
Since $\bg^*$ is admissible for our optimal control problem, see \eqref{z6n},
we now deduce from  \eqref{ii} that $\bg^*$ is a solution for the problem $\OO$, which concludes the proof.
\end{proof}

{We end this section with the remark  that an example of a function $j$ which satisfies condition \eqref{z1} is given by
\begin{equation}j(\bv)=\alpha\,\|\bv(T)-\bu_d\|_V^2\qquad\forall\,\bv\in C([0,T];V),\label{jh}
\end{equation}
where $\alpha>0$ and $\bu_d$ is a given element in $V$.  Consider now the contact model presented in Section \ref{s2}
and recall that in this case the function $\fb$ is given by \eqref{cop}. Then, condition \eqref{z11a} is satisfied  if   
\begin{equation}\label{comp}
Z:=L^2(\Omega)^d \times L^2(\Gamma_2)^d,\quad	\dual{ \KK(\bp,\bs)(t)}{\bv}=\int_{\Omega} \bp(t)\bv\,dx+\int_{\Gamma_2}\bs(t)\bv\,da
\end{equation}
for all  $(\bp,\bs)\in H^1(0,T;Z)$, $t\in [0,T]$ and $\bv\in V$.
Moreover, if $\fb_0$ is a fixed given datum, then   \eqref{z11a} holds if
\begin{equation}\label{comp0}
Z:= L^2(\Gamma_2)^d,\quad\dual{ \KK(\bs)(t)}{\bv}=\int_{\Omega} \fb_0(t)\bv\,dx+\int_{\Gamma_2}\bs(t)\bv\,da
\end{equation}
for all $\bs\in H^1(0,T;Z)$,  $t\in [0,T]$ and $\bv\in V$. Note that, in the case of the choice \eqref{jh}, \eqref{comp}, the significance  of the Problem $\cO$ is the  following: we are looking for  a density of body forces and a density of surface tractions  such that the  displacement of the viscoelastic body  at the end of  the time interval of interest is as close as possible to a given displacement field $\bu_d$. In the case of the choice \eqref{jh}, \eqref{comp0}, we are looking for  a density of surface tractions such that the  displacement of the viscoelastic body at the end of the time interval of interest is as close as possible to a given displacement field $\bu_d$. } Furthermore, in both cases above, the choice has to fulfil a minimum expenses  condition,  taken into account by the second term in \eqref{z3}. In fact, a compromise policy between the two aims (``$\bu(T)$ close to $\bu_d$" and ``minimal expenses") has to be found and the relative importance of each criterion with respect to the other is expressed by the choice of the weight coefficients $\alpha,\, \beta>0$.

\section{Sensitivity analysis}\label{s4}
\setcounter{equation}0

This section is dedicated to the study of the differentiability properties with respect to the given data $\fb$ of the solution operator $S$ associated to the Problems $\cP$ and $\cQ$, see \eqref{so}. \medskip
Before we  start, let us recall the concept of (Hadamard) directional differentiability.

\begin{Definition}\label{def:dd}
	Let $X,Y$ be Banach spaces.
	
	\smallskip
	{\rm a)} We say that the operator $\rho:X\to Y$ is {\rm directionally differentiable} at $x \in X$ in direction $\delta x\in X$ if the following limit exists:
	\[\rho'(x;\delta x):=\lim_{\tau \searrow 0}\frac{\rho(x+\tau \delta x)-\rho(x)}{\tau} \quad\  \text{\rm in }\ Y.\]
	
	\smallskip
{\rm 	b)} The operator $\rho:X  \to Y$ is said to be {\rm  Hadamard directionally differentiable} at $x \in X$ in direction $\delta x\in X$ if the following limit exists:
	\begin{equation*}
		\rho'_H(x;\delta x):=\lim_{\substack{\tau \searrow 0 \\ z \to \delta x}} \frac{\rho(x+\tau z)-\rho(x)}{\tau} \quad\ \text{\rm in }\ Y.
	\end{equation*}
In this case, the operator $\rho'_H(x;\cdot):X\to Y$ is called Hadamard directional derivative of $\rho$ at $x$.

\smallskip
{\rm c)} Let $Y_0$ be a Banach space so that $Y$ is continuously embedded in $Y_0$. We say that $\rho$ is Hadamard directionally differentiable from $X$ to $Y_0$,  if  the  difference quotient  defined in  b) converges only in the space $ Y_0$, for each $x,\delta x \in X$.
\end{Definition}

\begin{Remark}\label{rem:H}
If a mapping $\rho:X \to Y_0$ is Lipschitz continuous and directionally differentiable, then it is   Hadamard directionally differentiable. A proof of this result can be found in  \cite[Lemma 3.1.2b)]{schirotzek}.
\end{Remark}

In the rest of this section, we use the notation $M^\circ:=\{\bx^* \in V^* : \dual{\bx^*}{\bx} \leq 0 \quad \forall\, \bx \in M\}$ for the polar cone of a set $ M \subset V$. Given $\bx \in V^*$, we denote its annihilator by $ [\bx]^{\perp}:=\{\bmu \in V: \dual{\bx}{ \bmu}=0 \} $. Moreover, we abbreviate the set $\cup_{\alpha >0} \alpha M$ by $\R^+ M$. We also assume that the conditions \eqref{U}--\eqref{f} are satisfied in all what follows, even if we do not mention them  explicitly.

As already explained in the introduction, we shall make use of the fact that
$\cP$  can be equivalently rewritten as the fixed point problem $\cQ$, see Theorem \ref{t1n}. The problem $\cQ$ involves the solution operator of the variational inequality \eqref{eq} and to ensure its directional differentiability we need the following additional assumptions on the given data $U$ and $P$:
\begin{eqnarray}
&&\label{Up} \left\{\begin{array}{ll}
\mbox{ $U$ is such that}\\ [2mm]	
\mbox{
	$\overline{\R_+(U-\bz)} \cap [\bzeta]^{\perp} =\overline{\R_+(U-\bz) \cap [\bzeta]^{\perp}}\quad \, \forall\,(\bz,\bzeta) \in U \times \overline{\R^+( U-\bz)}^\circ, $
}\end{array}
\right.
\\ [4mm]
&&\label{Pd}\mbox{ $P: V\to V^*$ is directionally differentiable}.
\end{eqnarray}

\begin{Remark}\label{rem_p}

Note that the condition imposed on $U$ in \eqref{Up} is also known as polyhedricity \cite[Sec 6.4]{bs} and it plays an essential role when it comes to the differentiability of variational inequalities of first kind \cite{haraux,mp76}.

Let us emphasize that both \eqref{Up} and \eqref{Pd} are satisfied in the setting of the mechanical problem described in Section \ref{s2}. To see this, we recall that, in this particular framework, $U$ is given by \eqref{eq:convexset}. The fact that the set from  \eqref{eq:convexset} is polyhedric is a well-known fact, for a rigorous proof we refer to \cite[Thm.\ 3.36]{diss}.

Concerning the condition imposed on $P$ in \eqref{Pd}, we note that it is satisfied by the operator \eqref{pp} with $p$ as in \eqref{G7n}. This is due to standard arguments, based on assumption \eqref{G7n}$(f)$ and the Lebesgue dominated convergence theorem combined with the condition \eqref{G7n}$(b)$ and \eqref{trace},  as shown in \cite{susu_dam}, for instance.

\end{Remark}

Our first result in this section is the following.

\begin{Lemma}\label{lem:dd}Assume that \eqref{Up}--\eqref{Pd} are satisfied. Then, the operator $\FF:V^* \to V$ is directionally differentiable at $\bomega$ in any direction $\dom $. Moreover, the element $\dz:=\FF'(\bomega;\dom)$ is the unique solution of the variational inequality
	\begin{eqnarray}\label{eq:f_derivv}
	&&\dz \in  {\CC(\bz,\bzeta)},\ \  (\BB\bvarepsilon(\dz),\bvarepsilon(\beeta)-\bvarepsilon(\dz))_Q\\[2mm]
	&& \qquad\qquad+\dual{P'(\bz;\dz)}{\beeta-\dz}_V\ \geq \dual{\delta \bomega}{\beeta-\dz}_Y  \quad \forall\, \beeta \in {\CC(\bz,\bzeta)},\nonumber
	\end{eqnarray} where $\bz:=\FF(\bomega), \bzeta:=\bomega-(\WW+P)(\bz)$ and $\CC(\bz,\bzeta):=\overline{\R_+(U-\bz)} \cap [\bzeta]^{\perp}$.\end{Lemma}

\medskip

\begin{proof}Since the operator $\WW+P:V \to V^*$ is directionally differentiable, cf.\,\eqref{Pd}, and since
the set $U\subset V$ is polyhedric, by assumption \eqref{Up}, we can apply \cite[Prop.\,11, Thm.\,12]{w0}.
This tells us that the operator
	$\FF:V^* \to V$ is directionally differentiable  and the element $\dz:=\FF'(\bomega;\dom)$ is the unique solution to the variational inequality
	\begin{eqnarray}
	&&\dz \in  {\CC(\bz,\bzeta)},\ \ \dual{(\WW+P)'(\bz;\dz)}{\beeta-\dz}_V \geq \dual{\delta \bomega}{\beeta-\dz}_Y  \quad \forall\, \beeta \in  {\CC(\bz,\bzeta)}.\nonumber
\end{eqnarray}
We now use the definition of the operator $\WW$ from \eqref{ww} to conclude the proof.
\end{proof}

We  proceed with the following existence and uniqueness result.

\begin{Lemma}\label{ex}
	Let $\bh, \df \in C([0,T];V^*)$ be given. Then there exists a unique function $\dbu \in L^\infty(0,T;V)$
	such that
	\begin{equation}\label{dd} \dbu(t)
		=\FF'\Big[\bh(t); \df(t)- \bvarepsilon^*\Big(\int_0^t \RR(t-s) \bvarepsilon (\dbu(s))\,ds\Big)\Big] \quad \mbox{\rm   a.e. }\  t\in (0,T). \end{equation}
	
\end{Lemma}
\

\begin{proof}	Since $\FF:V\to V^*$ is a Lipschitz continuous operator, we have  that  $\FF'(\bh(t);\cdot):V^* \to V$ is Lipschitz continuous too, with the same Lipschitz constant, denoted by $L_{\FF}:=\frac{1}{m_\BB},$ see \eqref{Fl}. We now use assumption \eqref{R} to see that  the mapping
	\[L^1(0,T;V)\ni \beeta \mapsto \bvarepsilon^*\Big(\int_0^\cdot \RR(\cdot-s) \bvarepsilon (\beeta(s)))\,ds\Big)  \in C([0,T];V^*)\]
	is well-defined. It follows from here that the operator
	\[\beeta \mapsto \GG(\beeta):=\FF'\Big[\bh(\cdot); \df(\cdot)- \bvarepsilon^*\Big(\int_0^\cdot \RR(\cdot-s) \bvarepsilon(\beeta(s))\,ds\Big)\Big]\]
	maps $L^1(0,T;V)$ to $L^\infty(0,T;V)$.

	 Next, we remark that, for any elements  $\beeta_1,\beeta_2\in L^\infty(0,T;V)$,
	 the following inequality holds:
	 \begin{eqnarray*}
		&&\|\GG(\beeta_1)(t)-\GG(\beeta_2)(t)\|_{V}\leq \frac{1}{m_\BB}\Big\|\bvarepsilon^*\Big(\int_0^t \RR(t-s) \bvarepsilon ((\beeta_1-\beeta_2)(s))\,ds\Big)\Big\|_{V^*}
		\\
		&&\qquad\leq \frac{d}{m_\BB}\|\RR\|_{C([0,T];{\bf Q_\infty)}} T \|\beeta_1-\beeta_2\|_{L^\infty(0,T;V)}.
	\end{eqnarray*}	
This means that if $ \frac{d}{m_\BB}\|\RR\|_{C([0,T];{\bf Q_\infty)}} T<1$, then the
operator $\GG:L^\infty(0,T;V) \to L^\infty(0,T;V)$ is a contraction and, using the Banach fixed point argument, we deduce  that the nonlinear equation \eqref{dd} admits a unique solution in $L^\infty(0,T;V)$.

Nevertheless,  since $T>0$ does not necessarily satisfy the above smallness condition, we need to resort to a concatenation argument to  prove the unique solvability of \eqref{dd} on the whole interval $(0,T)$. To this end, we define $T^*:=\frac{m_\BB}{2d\|\RR\|_{C([0,T];{\bf Q_\infty)}}}$  and note that
\begin{equation}\label{tstar} \frac{d}{m_\BB}\|\RR\|_{C([0,T];{\bf Q_\infty)}} T^*<1.
\end{equation}
Then, it follows from above that equation
\[\beeta=\GG_1(\beeta) \quad \text{in }(0,T^*)\]
has a unique solution, denoted in what follows by $\dbu^1$. Note that here we use the notation
$\GG_1$ for the restriction of the operator $\GG$ to the space  $L^\infty(0,T^*;V)$, that is $\GG_1:=\GG|_{L^\infty(0,T^*;V)}$.

In the next step, we consider the mapping $\GG_2:L^\infty(T^*,2T^*;V)\to L^\infty(T^*,2T^*;V)$ given by
	\small
	\[\GG_2(\beeta)(t):=\FF'\Big[\bh(t); \df(t)- \bvarepsilon^*\Big(\int_0^{T^*} \RR(t-s) \bvarepsilon(\dbu^1(s))\,ds\Big)-\bvarepsilon^*\Big(\int_{T^*}^t \RR(t-s) \bvarepsilon(\beeta(s))\,ds \Big)\Big] \]\normalsize
	for all $t \in (T^*,2T^*)$.  Then, arguments similar to those used above show that the operator $\GG_2$ is well defined and, for any  $\beeta_1,\ \beeta_2\in L^\infty(T^*,2T^*;V)$,  the following inequality holds:
\begin{eqnarray*}
	&&\|\GG_2(\beeta_1)(t)-\GG_2(\beeta_2)(t)\|_{V}\leq \frac{1}{m_\BB}\Big\|\bvarepsilon^*\Big(\int_{T^*}^t \RR(t-s) \bvarepsilon ((\beeta_1-\beeta_2)(s))\,ds\Big)\Big\|_{V^*}
		\\[2mm]
	&&\qquad\leq \frac{d}{m_\BB}\|\RR\|_{C([0,T];{\bf Q_\infty)}} T^* \|\beeta_1-\beeta_2\|_{L^\infty(T^*,2T^*;V)}.
	\end{eqnarray*}
Then, using \eqref{tstar} it follows that the mapping  $\GG_2:L^\infty(T^*,2T^*;V)\to L^\infty(T^*,2T^*;V)$ is a contraction. Hence,  using again the Banach fixed point principle, it follows that the equation
\[\beeta=\GG_2(\beeta) \quad \text{in }(T^*,2T^*)\]
has  a unique solution in $L^\infty(T^*,2T^*;V)$. Denote this solution by  $\dbu^2$.

The procedure can be continued in this manner until, after a finite number of steps $K$, we have $K T^* >T.$ We set $K:=[T/T^*]+1$ and, for simplicity, we do not make the difference between $K T^*$ and $T$. We denote by $\dbu^k$, $k=1,..,K,$ the unique solution to
	\[\beeta=\GG_k(\beeta) \quad \text{in }((k-1)T^*,kT^*),\]
	where
  \small
	\begin{eqnarray*}
		&&\hspace{-5mm}\GG_k(\beeta)(t)\\ [2mm]
		&&\hspace{-5mm}:=\FF'\Big[\bh(t); \df(t)- \sum_{i=1}^{k-1} \bvarepsilon^*\Big(\int_{(i-1)T^*}^{iT^*} \RR(t-s) \bvarepsilon (\dbu^i(s))\,ds\Big)-\bvarepsilon^*\Big(\int_{(k-1)T^*}^t \RR(t-s) \bvarepsilon (\beeta(s))\,ds\Big) \Big]
	\end{eqnarray*}\normalsize
	for all $t \in ((k-1)T^*,kT^*).$
	
	We now prove that the solution that solves \eqref{dd} can be obtained by concatenating $\dbu^k,$ $k=1,..,K$. To this end, we define
	\begin{equation}\label{duf}\dbu_{fin}(t):=\sum_{k=1}^{K} \chi_{((k-1)T^*,kT^*)}  \dbu^k(t) \quad\  \forall\, t \in [0,T]\end{equation}
	where, here and below, we use the notation $\chi_A$ for the characteristic function of the set $A\subset \real$, that is
	\begin{equation*}\label{}
		\chi_A(t)=\left\{\begin{array}{ll} 0\qquad {\rm if}\ \ t\in A\\ [2mm]
		1\qquad{\rm otherwise}.
	\end{array}\right.
	\end{equation*}
	Note that, since $\dbu^k\in L^\infty((k-1)T^*,kT^*;V)$ for $k=1,..,K$, we obtain that $\dbu_{fin}\in L^\infty(0,T;V)$.
	Next, we consider an arbitrary element  $t \in (0,T)$ and note that there exists $k$ so that $t \in ((k-1)T^*,kT^*)$. Then,
	
	\small{\begin{align*}&\dbu_{fin}(t)
			\\&\overset{\eqref{duf}}{=} \dbu^k(t)
			\\&=\GG_k(\dbu^k)(t)
			\\&=\FF'\Big[\bh(t); \df(t)- \sum_{i=1}^{k-1} \bvarepsilon^*\Big(\int_{(i-1)T^*}^{iT^*} \RR(t-s) \bvarepsilon (\dbu^i(s))\,ds\Big)-\bvarepsilon^*\Big(\int_{(k-1)T^*}^t \RR(t-s) \bvarepsilon (\dbu^k(s))\Big)\,ds \Big]
			\\&\overset{\eqref{duf}}{=}\FF'\Big[\bh(t); \df(t)- \sum_{i=1}^{k-1} \bvarepsilon^*\Big(\int_{(i-1)T^*}^{iT^*} \RR(t-s) \bvarepsilon (\dbu_{fin}(s))\,ds\Big)-\bvarepsilon^*\Big(\int_{(k-1)T^*}^t \RR(t-s) \bvarepsilon (\dbu_{fin}(s))\,ds \Big)\Big]
			\\&=\FF'\Big[\bh(t); \df(t)- \bvarepsilon^*\Big(\int_{0}^t \RR(t-s) \bvarepsilon (\dbu_{fin}(s))\,ds \Big)\Big]
			\\&=\GG(\dbu_{fin})(t).
	\end{align*}}\normalsize
We conclude from above that
\begin{equation*}
	\dbu_{fin}=\GG(\dbu_{fin}) \quad \text{in }(0,T),
\end{equation*}
which proves that  the function $\dbu_{fin} \in L^\infty(0,T;V)$ is a solution for \eqref{dd}.
	The uniqueness of  the solution $\dbu_{fin}$ follows by using the
	Lipschitz continuity of the $\FF'(\bh(t);\cdot)$, the assumption \eqref{R} on $\cR$ and   Gronwall's inequality.
\end{proof}

\begin{Remark}\label{rem}
Note that the function $\dbu $ is not expected to be continuous in time. Indeed, even if $\bh,\df$ are very smooth in time,  the smoothness of the right hand side in \eqref{dd} does not get improved, as $\FF'(\cdot;\cdot)$ is not continuous with respect to the point of derivation, i.e., at $\bv \in V^*$, the mapping $\FF'(\cdot;\bv)$ is not necessarily continuous.
\end{Remark}

We are now in a position to state and prove the main result of this section.
Note that, here and below, for each $t\in[0,T]$, the following notations are used:
 \begin{eqnarray*}
 &&\bzeta(t):= \fb(t)-\bvarepsilon^*\Big(\int_0^t \RR(t-s) \bvarepsilon \bu(s)\,ds\Big)-(\WW+P)(\bu(t)),\\ [2mm]
 &&\CC(\bu(t),\bzeta(t)):=\overline{\R_+(U-\bu(t))} \cap [\bzeta(t)]^{\perp},\nonumber
\end{eqnarray*}where $\bu=S(\fb).$

\begin{Theorem}\label{t3} Assume \eqref{U}--\eqref{f} and \eqref{Up}--\eqref{Pd}. Then,
	the solution operator \eqref{so}  is Hadamard directionally differentiable from $C([0,T];V^*) $ to $L^\varrho(0,T;V)$ for each $\varrho \in (1,\infty)$.
	Moreover, its directional derivative at $ \fb$ in direction $ \df$, denoted by $\dbu:=S'(\fb;\df) \in L^\infty(0,T;V)$, is characterized as the unique solution of the history-dependent variational inequality
	\begin{eqnarray}
	&&\label{si_d}
		\dbu(t) \in {\CC(\bu(t),\bzeta(t))}, \quad (\BB \bvarepsilon (\dbu(t)), \bvarepsilon(\bv)- \bvarepsilon(\dbu(t)))_Q
			\\ [2mm]
	&&\quad+\dual{P'(\bu(t);\dbu(t))}{\bv-\dbu(t)}+\Big(\int_0^t \RR(t-s) \bvarepsilon (\dbu(s))\,ds, \bvarepsilon(\bv)- \bvarepsilon (\dbu(t))\Big)_Q
		\nonumber	\\ [2mm]
	&&\qquad  \geq \dual{ \df(t)}{\bv-\dbu(t)}  \quad\  \forall\, \bv \in {\CC(\bu(t),\bzeta(t))}, \ \mbox{\rm  a.e. }\  t\in (0,T).\nonumber
	\end{eqnarray}
\end{Theorem}

\begin{proof}
We only focus on proving the directional differentiability result. The Hadamard directional differentiability then follows immediately from the Lipschitz continuity of $S$   established in Lemma \ref{l2} and Remark \ref{rem:H}.
 By virtue of Lemma \ref{lem:dd}, the variational inequality \eqref{si_d} is equivalent to
\begin{equation}\label{dd0} \dbu(t)
		=\FF'\Big[\bh(t);\underbrace{ \df(t)- \bvarepsilon^*\Big(\int_0^t \RR(t-s) \bvarepsilon (\dbu(s))\,ds\Big)}_{=:\dbh(t)}\Big] \quad \mbox{ a.e. }t \in (0,T),
	\end{equation}
where, for simplicity, here and below, we use the notation
\[\bh(t):= \fb(t)-\bvarepsilon^*\Big(\int_0^t \RR(t-s) \bvarepsilon(\bu(s))\,ds\Big)\quad\forall\, t\in[0,T].\]
Note that,
according to Lemma \ref{ex}, equation \eqref{dd0} admits a unique solution $\dbu \in L^\infty(0,T;V)$.
Let $\tau>0$ be arbitrary, but fixed and denote $\bu_\tau:=S( \fb+\tau  \df)$.
Our aim in what follows is to prove that
	\begin{equation}\label{rez}
			\Big\| \frac{\bu_\tau-\bu}{\tau}-\dbu \Big\|_{L^\varrho(0,T;V)}
			\to 0\quad \ \text{as }\tau \searrow 0.
	\end{equation}
To this end, we fix $t\in[0,T]$ and use   Theorem \ref{t1n} to write  \begin{eqnarray*}
&&\bu(t)=\FF\underbrace{\Big[ \fb(t)-\eps^*\Big(\int_0^t  \RR(t-s)\eps(\bu(s))\,ds\Big)\Big]}_{=\bh(t)} \quad \text{ in }V^*,
			\\[2mm]
&&\bu_\tau(t)=\FF\underbrace{\Big[ \fb(t)+\tau  \df(t)- \eps^*\Big(\int_0^t \RR(t-s) \bvarepsilon (\bu_\tau(s))\,ds\Big)\Big]}_{=:\bh_\tau(t)} \quad \text{ in }V^*.
\end{eqnarray*}
\normalsize Thus,
\begin{eqnarray}\label{w}
&&\Big(\frac{\bu_\tau-\bu}{\tau}-\dbu \Big)(t)=\frac{\FF(\bh_\tau(t)) -\FF(\bh(t)+\tau \dbh(t))}{\tau}\\[2mm]
&&+\underbrace{\Big[\frac{\FF(\bh(t)+\tau \dbh(t))-\FF(\bh(t))}{\tau}-\FF'(\bh(t);\dbh(t))\Big]}_{=:r_{\FF,\tau}(t)}.\nonumber
\end{eqnarray}

Note that, due to the directional differentiability of $\FF:V^* \to V$ established in Lemma \ref{lem:dd}, we have
\begin{equation}\label{Z1}
r_{\FF,\tau}(t) \to 0 \  \text{ in } V \quad\  \text{as }\tau \searrow 0 .
\end{equation}
Moreover, as $\FF$ is Lipschitz continuous, it follows that
\begin{equation}\label{Z2}
	\|r_{\FF,\tau}(t)\|_V \le 2L_\FF\|\dbh(t)\|_{V^*} .
\end{equation}
The properties \eqref{Z1} and \eqref{Z2} allow us to apply Lebesgue's dominated convergence theorem in order to see that
	\begin{equation}
		\label{rgt0}r_{\FF,\tau} \to 0 \ \text{ in } L^{\varrho}(0,T;V) \quad \text{as }\tau \searrow 0,\quad \forall\,\varrho \in[1,\infty).
	\end{equation}

Next, using the Lipschitz continuity  of $\FF$, see \eqref{Fl},  and identity \eqref{w}, we find that
\begin{eqnarray}\label{above}
	&&\Big\| \Big(\frac{\bu_\tau-\bu}{\tau}-\dbu \Big)(t)\Big\|_V\\ [2mm]
		&&\qquad \leq \frac{1}{m_\BB}
			\Big\|\frac{\bh_\tau(t) -\bh(t)}{\tau}-\dbh(t)\Big\|_{V^*}+\|r_{\FF,\tau}(t)\|_{V}. \nonumber
\end{eqnarray}
	Note that
	\begin{equation*}
			\frac{\bh_\tau(t) -\bh(t)}{\tau}-\dbh(t)=-  \bvarepsilon^* \int_0^t  \RR(t-s) \bvarepsilon \Big(\frac{\bu_\tau(s)-\bu(s)}{\tau}-\dbu (s)\Big)\,ds
	\end{equation*}
	and, therefore,
	\begin{equation}\label{abov}
			\Big\| \frac{\bh_\tau(t) -\bh(t)}{\tau}-\dbh(t)\Big\|_{V^*}\leq d \|\RR\|_{C([0,T];{\bf Q_\infty)}} \int_0^t \Big\|\frac{\bu_\tau(s)-\bu(s)}{\tau}-\dbu(s)\Big\|_V\,ds.
\end{equation}
	We now combine the inequalities  \eqref{above}  and
	\eqref{abov} to deduce that
	\begin{equation*}
			\Big\| \Big(\frac{\bu_\tau-\bu}{\tau}-\dbu \Big)(t)\Big\|_V \leq
			\frac{d}{m_\BB} \|\RR\|_{C([0,T];{\bf Q_\infty)}} \int_0^t \Big\|\frac{\bu_\tau(s)-\bu(s)}{\tau}-\dbu(s)\Big\|_V\,ds+\| r_{\FF,\tau}(t)\|_{V} .
	\end{equation*}
Then,  applying Gronwall's inequality we find that
	\begin{eqnarray*}
			&&\Big\| \Big(\frac{\bu_\tau-\bu}{\tau}-\dbu \Big)(t)\Big\|_V\leq  \| r_{\FF,\tau}(t)\|_{V} +c\int_0^t e^{c(t-s)}\| r_{\FF,\tau}(s)\|_{V}\,ds
			\\ [2mm]
			&&\qquad \leq   \| r_{\FF,\tau}(t)\|_{V} +c\Big(\int_0^t e^{c(t-s)\varrho'}\,ds \Big)^{1/\varrho'} \| r_{\FF,\tau}\|_{L^\varrho(0,T;V)}
			\\[2mm]
			&&\qquad\qquad \leq \| r_{\FF,\tau}(t)\|_{V} +c\, \Big(\frac{ e^{c\varrho't}-1}{c\varrho'}\Big)^{1/\varrho'} \| r_{\FF,\tau}\|_{L^\varrho(0,T;V)} \quad  \forall\,t \in [0,T],
	\end{eqnarray*}
where $c:=\frac{d}{m_\BB} \|\RR\|_{C([0,T];{\bf Q_\infty)}}$ and $\varrho'$ is the conjugate exponent to $\varrho$, i.e.,
	\[1/\varrho+1/\varrho'=1.\]
	
We conclude from above that
\begin{eqnarray*}
			&&\Big\| \Big(\frac{\bu_\tau-\bu}{\tau}-\dbu \Big)\Big\|_{L^\varrho(0,T;V)}
			\\ [2mm]
			&&\qquad\leq  \| r_{\FF,\tau}\|_{L^\varrho(0,T;V)}+c T^{1/\varrho} \Big(\frac{ e^{c\varrho'T}-1}{c\varrho'}\Big)^{1/\varrho'} \| r_{\FF,\tau}\|_{L^\varrho(0,T;V)}\to 0\quad \text{as }\tau \searrow 0,
\end{eqnarray*}
and, using \eqref{rgt0} we deduce that the convergence \eqref{rez} holds.
Thus, we have proven that $S:C([0,T];V^*) \to L^\varrho(0,T;V)$ is directionally differentiable with derivative at $\fb$ in direction $\df$ given by $\dbu \in L^\infty(0,T;V)$. This is the unique solution to  equation \eqref{dd0} which, in fact, is equivalent to inequality \eqref{si_d}, as mentioned
at the beginning of the proof. We can now conclude the desired Hadamard directional differentiability, which completes the proof. \end{proof}

{We now go back to to the mechanical framework described at the end of Section \ref{s2}, in which $\fb$ is replaced by \eqref{cop}.
We claim that, in this particular case, Theorem \ref{t3} can be used to provide a sensitivity analysis result of $S$ with respect to the data $(\fb_0,\fb_2)$.  To this end, we note that the  operator $\KK:H^1(0,T; L^2(\Omega)^d \times  L^2(\Gamma_2)^d) \to C([0,T];V^*)$ from \eqref{comp} is linear and continuous and,
therefore, Fr\'echet-differentiable. Next, by chain rule combined with Theorem \ref{t3},
we have that the operator
$$(\fb_0,\fb_2)\mapsto S(\fb_0,\fb_2) $$is Hadamard directionally differentiable from $H^1(0,T; L^2(\Omega)^d \times  L^2(\Gamma_2)^d)$ to $L^\varrho(0,T;V)$ for each $\varrho \in (1,\infty)$, which concludes the proof of the claim.}

\section{Well-posedness results}\label{s5}
\setcounter{equation}0

Well-posedness concepts for nonlinear problems represent an important topic in Functional Analysis. This  has known a significant development in the last decades. Originating in the papers of  Tykhonov \cite{Ty} and Levitin-Polyak \cite{LP} (where the well-posedness of minimization problems was considered), well-posedness concepts have been extended to a large number of problems, including nonlinear equations, inequality problems, inclusions, fixed point problems, and optimal control problems. In particular, the well-posedness  of variational inequalities was studied for the first time in \cite{LP1,LP2}. Comprehensive references in the field are   the books \cite{DZ, L} and, more recently,
\cite{S}.

The well-posedness concepts depend on the problem considered, vary from author to author, and even from paper to paper. Nevertheless, most of these concepts are
based on two main ingredients: the existence and uniqueness of the solution to the corresponding problem and the convergence of a special class of sequences to it, the so-called approximating sequences.

In this section, we are working under assumptions \eqref{U}--\eqref{f}, even if we do not mention them explicitly. Then, using Theorems \ref{t1} and \ref{t1n}, it follows that Problems $\cP$ and $\cQ$ have a unique common solution, $\bu\in C([0,T];U)$.  For this reason, any well-posedness concept  we  consider will concern  simultaneously  Problems $\cP$ and $Q$, and will be determined by its set of approximating sequences. We start with the following definitions.

\begin{Definition}\label{dP} {\rm a)} A sequence $\{\bu_n\}\subset C([0,T];V)$ is said to be a $p$-approximating sequence  if
	$\bu_n(t)\in U$ for any $n\in\mathbb{N}$, $t\in[0,T]$, and
	if there exists  a sequence $0\le\ve_n\to 0$ such that
	\begin{eqnarray}
		&&\label{sin} \hspace{-9mm}(\BB\bvarepsilon({\bu}_n(t)),\bvarepsilon(\bv)-\bvarepsilon(\bu_n(t)))_Q+\Big(\int_0^t\RR(t-s)\bvarepsilon({\bu}_n(s))\,ds,\bvarepsilon(\bv)-\bvarepsilon(\bu_n(t))\Big)_Q
		\\[2mm]
		&&\hspace{-9mm}\qquad+\dual{P\bu_n(t)}{\bv-\bu_n(t)}+\ve_n\|\bv-\bu_n(t)\|_V\nonumber\\ [2mm]
		&&\hspace{-9mm}\qquad\qquad\geq
		\dual{ \fb(t)}{\bv-\bu_n(t)}\quad \forall\,\bv\in U,\, t\in[0,T],\, n\in\mathbb{N}.\nonumber
	\end{eqnarray}
	
\smallskip
	{\rm b)}
		Problems $\cP$ and $\cQ$ are said to be $p$-well-posed if  every $p$-approximating sequence converges in $C([0,T];V)$ to $\bu$.
	\end{Definition}

\begin{Definition}\label{dQ} {\rm a)} A sequence $\{\bu_n\}\subset C([0,T];V)$ is said to be a $q$-approximating sequence  if
	there exists  a sequence $0\le\ve_n\to 0$ such that
	\begin{equation}\label{fpn}
\Big\|	\bu_n(t)-\FF\Big[ \fb(t)- \eps^*\Big(\int_0^t  \RR(t-s)\eps(\bu_n(s))\,ds\Big)\Big]\Big\|_V\le\ve_n\quad \forall\, t\in[0,T],\, n\in\mathbb{N}.
	\end{equation}

\smallskip
	{\rm b)}
	Problems $\cP$ and $\cQ$ are said to be $q$-well-posed if  every $q$-approximating sequence converges in $C([0,T];V)$ to $\bu$.
\end{Definition}

 Our main result in this section is the following.

\begin{Theorem}\label{t4} Assume  \eqref{U}--\eqref{f}. Then, Problems $\cP$ and $\cQ$ are both $p$- and $q$-well-posed.
\end{Theorem}

\begin{proof} Let $\{\bu_n\}$ be a $p$-approximating sequence.
We fix $n\in\mathbb{N}$ and $t\in[0,T]$. We take $\bv=\bu_n(t)$ in \eqref{si}, $\bv=\bu(t)$ in \eqref{sin} and add the resulting inequalities to find that
\begin{eqnarray*}
	&&(\BB\bvarepsilon({\bu}_n(t))-\BB\bvarepsilon(
	\bu(t))),\bvarepsilon(\bu_n(t))-\bvarepsilon(\bu(t)))_Q\\ [2mm]
	&&\quad\le \Big(\int_0^t\RR(t-s)\big(\bvarepsilon({\bu}_n(s))-\bvarepsilon({\bu}(s))\big)
	\,ds,\bvarepsilon(\bu(t))-\bvarepsilon(\bu_n(t))\Big)_Q
	\\[2mm]
	&&\qquad+\dual{P\bu_n(t)-P\bu(t)}{\bu(t)-\bu_n(t)}+\ve_n\|\bu(t)-\bu_n(t)\|_V.
	\nonumber
\end{eqnarray*}
We now use assumptions \eqref{B} and \eqref{P} as well as  equality \eqref{684} to see that
\begin{eqnarray*}
&&m_{\cal B}\|\bu_n(t)-\bu(t)\|^2_V\\ [2mm]
&&\ \ \le\Big(\int_0^t
\|\RR(t-s)\big(\bvarepsilon({\bu}_n(s))-\bvarepsilon({\bu}_n(s))\big)\|_Q\,ds\Big)
\|\bu_n(t)-\bu(t)\|_V+\ve_n\|\bu(t)-\bu_n(t)\|_V.
\end{eqnarray*}
Next, by assumption \eqref{R} and inequality \eqref{pmp}  we
deduce that
\begin{equation*}
\|\bu_n(t)-\bu(t)\|_V\le c \int_0^t\|\bu_n(s)-\bu(s)\|_V\,ds+\frac{\ve_n}{m_{\cal B}}.
\end{equation*}
with $c=\frac{d}{m_{\mathcal B}}\, \|\mathcal R\|_{C([0,T];{\bf Q_\infty})}$. Finally, we use Gronwall's inequality to find that
\[\|\bu_n(t)-\bu(t)\|_V\le \frac{\ve_n}{m_{\cal B}\,}e^{ct}\le \frac{\ve_n}{m_{\cal B}}\,e^{cT}.   \]
Therefore,  the convergence $\ve_n\to 0$ implies that $\bu_n\to \bu$ in $C([0,T];V)$ which shows that Problems $\cP$ and $\cQ$ are $p$-well-posed.

\medskip

Assume now that $\{\bu_n\}$ is a $q$-approximating sequence.
We use inequality \eqref{fpn} and definition \eqref{la} of the operator $\Lambda$ to  see that
\begin{equation*}
\|\bu_n(t)-\Lambda\bu_n(t)\|_V\le\ve_n \quad\ \forall\, t\in [0,T],\, n\in\mathbb{N}
\end{equation*}
and, since $\ve_n\to 0$,  we find that
\begin{equation}\label{71}
\|\bu_n-\Lambda\bu_n\|_{C([0,T];V)}\to 0.
\end{equation}

On the other hand, arguments similar to those used in Lemma \ref{l2} show that $\Lambda$ is a history-dependent operator, that is, it satisfies the inequality
\begin{equation}\label{72}
	\|\Lambda \bv(t)-\Lambda\bw(t)\|_V\le L\int_0^t\
	\|\bv(s)-\bw(s)\|_V\,ds\qquad\forall\,\bv,\, \bw\in C([0,T];V),\ t\in[0,T]
\end{equation}
with some constant $L>0$. Denote by $\Lambda^k$ the powers of the operator $\Lambda$, for $k=1,2,\ldots$.
Then,  using inequality \eqref{72}, after some algebraic calculations,  we deduce that the operator $\Lambda^k:C([0,T];V)\to C([0,T];V)$ is a Lipschitz continuous operator with constant $L_k=\frac{L^kT^{k}}{k!}$, that is,
\begin{equation}\label{34n}
	\|\Lambda^k\bv-\Lambda^k\bw\|_{C([0,T];V)}\le
	L_k\|\bv-\bw\|_{C([0,T];V)}\quad\forall\, \bv,\, \bw\in C([0,T];V).
\end{equation}
Next, it is easy to see that
\[\lim_{k\to\infty}\,\frac{L^kT^{k}}{k!}=0\]
and, therefore,  inequality \eqref{34n}  guarantees  that there exists $p\in\mathbb{N}$ such that  $\Lambda^p$ is a contraction on the space  $C([0,T];V)$ i.e.,
\begin{equation}\label{34}
	L_{p}<1.
\end{equation}
We now use equality $\Lambda^p\bu =\bu$ to see that
\begin{eqnarray*}
	&&\|\bu_n-\bu\|_{C([0,T];V)} \\ [2mm]
	&&\ \ \le \|\bu_n-\Lambda \bu_n\|_{C([0,T];V)}+\|\Lambda \bu_n-\Lambda^2 \bu_n\|_{C([0,T];V)}\\ [2mm]
	&&\ \ \ +\ldots+\|\Lambda^{p-1} \bu_n-\Lambda^{p}\bu_n\|_{C([0,T];V)}+\|\Lambda^p \bu_n-\Lambda^p \bu\|_{C([0,T];V)}.
\end{eqnarray*}
Next, using \eqref{34n} we
find that
\begin{eqnarray*}
	&&\|\bu_n-\bu\|_{C([0,T];V)}\\ [2mm]
	&&\quad\le \big(1+L_1+L_2+\ldots+L_{p-1}\big)
	\|\bu_n-\Lambda\bu_n\|_{C([0,T];V)}+ L_{p}\|\bu_n-\bu\|_{C([0,T];V)},
\end{eqnarray*}
which implies that
\begin{equation*}
	(1-L_{p})\|\bu_n-\bu\|_{C([0,T];V)}\le \big(1+L_1+L_2+\ldots+L_{p-1}\big)
	\|\bu_n-\Lambda\bu_n\|_{C([0,T];V)}.
\end{equation*}
Then,  inequality \eqref{34} yields
\begin{equation*}
	\|\bu_n-\bu\|_{C([0,T];V)}\le \frac{1}{1-L_{p}}\big(1+L_1+L_2+\ldots+L_{p-1}\big)
	\|\bu_n-\Lambda\bu_n\|_{C([0,T];V)}.
\end{equation*}
Finally, we use the convergence  \eqref{71} to see  that
$\bu_n\to \bu$ in $C([0,T];V)$. This shows that Problems $\cP$ and $\cQ$ are $q$-well-posed and concludes the proof of the theorem.
\end{proof}

We now proceed with some additional  results on the well-posedness of Problems $\cP$ and $\cQ$. To this end, we introduce the following notation:
\begin{eqnarray*}
	&&\hspace{-8mm}\cS=\Big\{\,   \{\bu_n\}\subset C([0,T];V)\ : \ \bu_n\to\bu\ \ {\rm in}\ \ C([0,T];V)\,\Big\},\\ [2mm]
	&&\hspace{-8mm}\cS_p=\Big\{\,   \{\bu_n\}\subset C([0,T];V)\ :   \{\bu_n\}\ \mbox{is a $p$-approximating sequence}\,\Big\}, \\ [2mm]
	&&\hspace{-8mm}\cS_q=\Big\{\,  \{\bu_n\}\subset C([0,T];V)\ :   \{\bu_n\}\ \mbox{is a $q$-approximating sequence}\,\Big\}.
\end{eqnarray*}	Then, Theorem \ref{t4} states that
$\cS_p\subset \cS$ and $\cS_q\subset \cS$. Our aim in what follows is to compare the concepts of $p$-well-posedness and $q$-well posedness introduced above. To this end, we shall use  Remark \ref{r6} and Example \ref{ex6} below to see that

\begin{equation}\label{61}
\cS_q=\cS, \qquad \cS_p\subset \cS_q,\qquad \cS_p\ne\cS_q.
\end{equation}

\begin{Remark}\label{r6} Assume that $\{\bu_n\}\in\cS$.
Then,  $\bu_n\to\bu$ in $C([0,T];V)$ and, since the operator $\Lambda$ defined by \eqref{la} is continuous we deduce that
$\bu_n-\Lambda\bu_n\to \bu-\Lambda\bu$ in $C([0,T];V)$. On the other hand, \eqref{ffp} guarantees that $\Lambda\bu=\bu$, which implies that $\bu_n-\Lambda\bu_n\to \bzero$ in $C([0,T];V)$. This shows that
$\{\bu_n\}\in\cS_q$. We conclude from here that $\cS\subset \cS_q$ and, since the converse inclusion was proved in Theorem $\ref{t4}$, we deduce that $\cS_q=\cS$. Moreover, Theorem $\ref{t4}$ shows that $\cS_p\subset\cS$ and, therefore, $\cS_p\subset\cS_q$.

\end{Remark}

\begin{Example}\label{ex6}
Consider the Problems $\cP$ and $\cQ$ in the particular setting of Example $\ref{ex0}$. Recall that the common solution of these problems is the function $u$ given by \eqref{sol}. Let $\{u_n\}$ be the sequence defined by
\begin{equation*}
	u_n(x,t)=xe^{-t}+\frac{1}{n}\qquad\forall\,(x,t)\in [0,1]\times[0,T], \ n\in\mathbb{N}.
\end{equation*}
Then, it is easy to see that $u_n\to u$ uniformly and, therefore, $\{u_n\}\in\cS=\cS_q$. Nevertheless, since $u_n(1,0)=1+\frac{1}{n}>1$, we deduce that condition $u_n(t)\in U$ for any $t\in[0,T]$ and $n\in\mathbb{N}$ is not satisfied.	 Therefore, $\{u_n\}\notin \cS_p$. This shows that there exist $q$-approximating sequences which are not $p$-approximating  sequences and, therefore, $ \cS_p\ne\cS_q$.
	
\end{Example}

	We end this section with the following comments on the equality $\cS_q=\cS$ in \eqref{61}.

\begin{Remark} We claim that among all the concepts which make Problems $\cP$  and $\cQ$ well-posed, the $q$-well-posedness
	concept in Definition $\ref{dQ}$ is optimal, in the sense that it uses the largest set of approximating sequences. Indeed, consider a different well-posedness concept, say the $r$-well-posedness concept, defined by a set of $r$-approximating sequences, denoted by
	$\cS_r$. Then, if Problems $\cP$ and $\cQ$ are $r$-well-posed we have 	$\cS_r\subset\cS $, by definition. Now,  since $\cS_q=\cS$, we deduce that $\cS_r\subset\cS_q$,  which justifies our claim. In particular, since the inclusion $\cS_p\subset \cS_q$ is strict, it follows from above that the $q$-well-posedness result in Theorem $\ref{t4}$ is stronger than the $p$-well-posedness result in the same theorem. As a consequence, the $q$-well-posedness concept in Definition $\ref{dQ}$ is better than the $p$-well-posedness concept in Definition $\ref{dP}$.
\end{Remark}

\begin{Remark}	Equality $\cS_q=\cS$  shows that a sequence $\{\bu_n\} \subset C([0,T];V)$ converges uniformly to the solution $\bu$
of Problems $\cP$ and $\cQ$ if and only if it is a $q$-approximating sequence, that is, if it satisfies the condition in Definition $\ref{dQ}$\,{\rm (a)}. It turns out from here that this  condition represents a {\rm criterion of convergence} to the common solution of Problems $\cP$ and $\cQ$.
\end{Remark}

\noindent

\section*{Aknowledgment}
{The work of Livia Betz was supported by the DFG grant BE 7178/3-1 for the project ``Optimal Control of Viscous
Fatigue Damage Models for Brittle Materials: Optimality Systems".
The work of Andaluzia Matei and Mircea Sofonea was supported by the European program {\it ACROSS} including the {\it University of Craiova}, Romania, and the {\it University of Perpignan Via Domitia}, France.}

\end{document}